\documentclass[11pt,a4paper]{elsarticle}
\usepackage{amsmath,amsthm,amsfonts,amssymb}
\usepackage{latexsym,euscript,dsfont}
\usepackage{xcolor}
\usepackage{txfonts}
\usepackage{mathrsfs}
\usepackage{bbm}
\usepackage{colortbl}
\usepackage{epstopdf}
\usepackage{graphicx}
\usepackage[style=base]{caption}
\usepackage{subcaption}
\usepackage{enumerate}
\usepackage{url}
\usepackage[pagewise]{lineno}

\theoremstyle{plain}
\newtheorem{thm}{Theorem}[section]
\newtheorem{dfn}[thm]{Definiton}
\newtheorem{prop}[thm]{Proposition}
\newtheorem{lem}[thm]{Lemma}
\newtheorem{cor}[thm]{Corollary}

\newtheorem{rem}{Remark}

\numberwithin{equation}{section}
\begin{document}
\title{On the distribution of the Cantor-integers}
\author[hzau]{Chun-Yun Cao}
\ead{caochunyun@mail.hzau.edu.cn}
\author[hzau]{Jie Yu}
\ead{yuj@webmail.hzau.edu.cn}

\address[hzau]{College of informatics, Huazhong Agricultural University, 430070 Wuhan, P.R.China.}

\begin{abstract}
For any positive integer $p\geq 3$, let $A$ be a proper subset of $\{0,1,\ldots, p-1\}$ with $\sharp A=s\geq 2$. Suppose $h: \{0,1,\ldots,s-1\}\to A$ is a one-to-one map which is strictly increasing with $A=\{h(0),h(1),\ldots,h(s-1)\}$.
We focus on so-called  Cantor-integers $\{a_n\}_{n\geq 1}$, which consist of these positive integers $n$ such that all the digits in the $p$-ary expansion of $n$ belong to $A$.
Let $\mathfrak{C}=\left\{\sum\limits_{n\geq 1}\frac{\varepsilon_n}{p^n}: \varepsilon_n\in A \text{ for any positive integer } n\right\}$  be the appropriate Cantor set, and denote the classic self-similar measure supported on $\mathfrak{C}$ by $\mu_{\mathfrak{C}}$.
Now that $n^{\log_s p}$ is the growth order of $a_n$  and $\left\{\frac{a_n}{n^{\log_s p}}:~n\geq 1\right\}'$ is precisely the set $\left\{\frac{x}{(\mu_{\mathfrak{C}}([0,x]))^{\log_s p}}: x\in\mathfrak{C}\cap[\frac{h(1)}{p},1]\right\}$, where $E'$ is the set of limit points of $E$, we show that $\left\{\frac{a_n}{n^{\log_s p}}:~n\geq 1\right\}'$ is just an interval $[m,M]$  with $m:=\inf\left\{\frac{a_n}{n^{\log_s p}}:n\geq 1\right\}$ and $M:=\sup\left\{\frac{a_n}{n^{\log_s p}}:n\geq 1\right\}$.
In particular, $\left\{\frac{x}{(\mu_{\mathfrak{C}}([0,x]))^{\log_s p}}: x\in\mathfrak{C}\backslash\{0\}\right\}=[m,M]$ if $0\in A$, and $m=\frac{q(s-1)+r}{p-1}, M=\frac{q(p-1)+pr}{p-1}$ if the set $A$ consists of all the integers in $\{0,1,\ldots, p-1\}$ which have the same remainder $r\in\{0,1,\ldots,q-1\}$ modulus $q$ for some positive integer $q \geq 2$ (i.e. $h(x)=qx+r$).
We further show that the sequence $\left\{\frac{a_n}{n^{\log_s p}}\right\}_{n\geq 1}$  is not uniformly distributed modulo 1, and it does not have the cumulative distribution function, but has the logarithmic distribution function (give by a specific Lebesgue integral).
\end{abstract}

\maketitle

\section{Introduction}%

The behavior of an arithmetical function $f(n)$ for large values of $n$  has always been one of the important problems in number theory. For example, the average order of a  fluctuating arithmetical function, the growth order of a monotone increasing arithmetical function. The distribution properties of the sequences derived from them have also been studied popularly.

An arithmetical function can be seen as a sequence. Let $\{r_n\}_{n\geq 0}$ be the Rudin-Shapiro sequence, the properties of the Rudin-Shapiro sums $s(n)=\sum_{k=0}^nr_k$ and $t(n)=\sum_{k=0}^n(-1)^kr_k$ have been developed by  Brillhart and  Morton in \cite{Brillhart78}, where it is showed that
\[\sqrt{\frac{3}{5}}<\frac{s(n)}{\sqrt{n}}<\sqrt{6} \qquad \text{and} \qquad 0\leq \frac{t(n)}{\sqrt{n}}<\sqrt{3},\]
for any $n\geq 1$, and that the sequences $\{s(n)/\sqrt{n}\}_{n\geq 1}$ and $\{t(n)/\sqrt{n}\}_{n\geq 1}$ are dense respectively in the intervals $[\sqrt{3/5},\sqrt{6}]$ and $[0,\sqrt{3}]$.
Five years later, in collaboration with  Erd\"{o}s, the authors further studied the distribution properties of the sequences $\{s(n)/\sqrt{n}\}_{n\geq 1}$ and $\{t(n)/\sqrt{n}\}_{n\geq 1}$ in \cite{Brillhart83}, they  showed that the sequences  do not have the cumulative distribution functions, but do have the logarithmic distribution functions at each point of the respective intervals $[\sqrt{3/5},\sqrt{6}]$ and $[0,\sqrt{3}]$.

In 2020, L\"{u}, Chen, Wen etc.\cite{Lv} introduced the quasi-linear integer sequence $f(n)$, which satisfies
\[\alpha:=\inf\left\{~\gamma\geq 0:~~\limsup_{n\to\infty}\frac{|f(n)|}{n^\gamma}=0~\right\}
>\beta:=\inf \left\{\gamma>0:~~\limsup_{n\to\infty} \frac{|f(n+1)-f(n)|}{n^\gamma}=0\right\},\]
and \[\left\{\frac{f(bn+i)-b^\alpha f(n)}{n^\beta}: n\geq 1, 0\leq i \leq b-1\right\} \text{ is bounded for some integer } b\geq 2.\]
They showed that the growth order of $f(n)$ is $n^\alpha$ (i. e. there exist $0<c_1\leq c_2$ such that $c_1\leq \frac{f(n)}{n^\alpha}\leq c_2$ for any $n$).
They focused on the limit function
\[\lambda(x):=\lim_{k\to\infty}\frac{f(b^kx)}{(b^kx)^\alpha}.\]
Their research shows that $\lambda$ is continuous, self-similar and bounded, further, $\{\lambda(x):x\geq 0\}$ is dense between any two limit points of the sequence $\left\{f(n)/n^\alpha \right\}_{n\geq 1}$. It is not hard to check that the Rudin-Shapiro sums $s(n)$ and $t(n)$ are both quasi-linear discrete functions.

For integers $p>s\geq 2$ and subset $A\subset\{0,1,\ldots,p-1\}$ with $\sharp A=s$.
We call an integer $n$ a Cantor integer if the digits in the  $p$-ary expansion of $n$ can only take values in $A$, which is named by analogy to the usual middle thirds Cantor set.

For convenience, let us introduce some notations which will be used throughout the text.
For any  integer $s\geq 2$ and any non-negative integer $n$, the  $s$-ary expansion of $n$ is denoted by
$$n=[\varepsilon_k \varepsilon_{k-1} \ldots \varepsilon_0]_s:=\sum_{i=0}^k\varepsilon_i s^i,$$ where the digit  $\varepsilon_i\in\{0,1,\ldots,s-1\}$ for any $0\leq i\leq k$. It would be noticed that the top digit $\varepsilon_k\neq 0$ if $n\neq 0$.
And for any $x\in(0,1)$, the  $s$-ary expansion of $x$ is denoted by
$$x=[0.~d_1 d_2 \ldots]_s:=\sum_{i=1}^\infty \frac{d_i} {s^i},$$
where the $s$-ary digit $d_i\in\{0,1,\ldots,s-1\}$ for any $i\geq 1$.
Thus for any real number $x\geq 0$, it can be written as
$x=[\varepsilon_k \varepsilon_{k-1} \ldots \varepsilon_0 ~.~d_1 d_2 \ldots]_s$, where $[\varepsilon_k \varepsilon_{k-1} \ldots \varepsilon_0]_s$ equals to $[x]$, known as the integer part of $x$, and $[0.~d_1 d_2 \ldots]_s$ equals to $\{x\}$, known as the fractional part of $x$.

Based on the notations above, suppose $h: \{0,1,\ldots,s-1\}\to A$ is a one-to-one map which is strictly increasing with $A=\{h(0),h(1),\ldots, h(s-1)\}$. Then for any positive integer $n=[\varepsilon_k \varepsilon_{k-1} \ldots \varepsilon_0]_s$, the $n$-th Cantor integer is
\[a_n=[h(\varepsilon_k) h(\varepsilon_{k-1}) \ldots h(\varepsilon_0)]_p.\]
Meanwhile, we adopt the convention that $a_0=0$. Note that
\[a_n=\sum_{i=0}^k h(\varepsilon_i) p^i =\sum_{i=0}^k \frac{h(\varepsilon_i)} {p^{k-i}}\cdot \left(\sum_{i=0}^k \frac{\varepsilon_i} {s^{k-i}}\right)^{-\log_s p}\cdot n^{\log_s p},\]
and $\sum_{i=0}^k \frac{h(\varepsilon_i)} {p^{k-i}}\in[1,p], \sum_{i=0}^k \frac{\varepsilon_i} {s^{k-i}}\in[1,s]$, we have that the growth order of $a_n$ is $n^{\log_s p}$. By further calculation, it could be obtained that $\alpha=\beta=\log_s p$ for the Cantor integers sequence. That is to say, the Cantor integers sequence what we concerned with are not quasi-linear discrete arithmetic functions.

Now, put
\begin{equation}\label{bn}
b_n:=\frac{a_n}{n^{\log_s p}}.
\end{equation}

The denseness of the sequence $\{b_n\}_{n\geq 1}$ defined in (\ref{bn}) has been discussed by Gawron and Ulas \cite{Gawron16} for the case $p=4, A=\{0,2\}$, and then by  Cao and Li \cite{Cao} for the case $p\geq 3, A=\{d:~~ d\in\{0,1,\ldots,p-1\},~ d \text{ is even}\}$. Moreover,  Cao and Li \cite{Cao}  established a connection between the Cantor integers and the self-similar measure. We will generalize their results and further investigate the distribution of the sequence $\{b_n\}_{n\geq 1}$, parallel to \cite{Brillhart83},  and  get some similar results about the Cantor integers. Moreover, we will consider whether or not the sequence $\{b_n\}_{n\geq 1}$ would be uniformly distributed modulo 1.

At first, we establish a close connection (in Section \ref{sec measure}) between the limit points of the sequence $\{b_n\}_{n\geq 1}$ and the self-similar probability measure $\mu_\mathfrak{C}$ which supported on the corresponding missing $p$-ary digit set
\[
\mathfrak{C}:=\left\{\sum_{n=1}^{\infty}\frac{\varepsilon_n}{p^n}: \varepsilon_n \in A \text{ for any positive integer } n\right\},
\]
with
\[\mu_\mathfrak{C}=\sum_{i=0}^{s-1}\frac{1}{s} \mu_\mathfrak{C}\circ S_i^{-1}, ~~\text{where} ~~S_i=\frac{x+h(i)}{p} ~~(i=0,1,\ldots, s-1).\]
\begin{thm}\label{thm measure}
Let $E'$ be the set of limit points of $E$, known as the derived set of $E$. One has
\[
\left\{\frac{x}{(\mu_\mathfrak{C}([0,x]))^{\log_s p}}: x\in\mathfrak{C}\cap\left[\frac{h(1)}{p},1\right]~\right\}=\{b_n:n\geq 1\}'.
\]
In particular, $\left\{\frac{x}{(\mu_\mathfrak{C}([0,x]))^{\log_s p}}: x\in\mathfrak{C}\backslash\{0\}~\right\}=\{b_n:n\geq 1\}'$ if $0\in A$ (i.e. $h(0)=0$).
\end{thm}

\vspace{10pt}

Then we discuss the limit points of the sequence $\{b_n\}_{n\geq 1}$, and show that each point between the supremum and infimum of $\{b_n\}_{n\geq 1}$ is an limit point of the sequence $\{b_n\}_{n\geq 1}$.

\begin{thm}\label{thm dense}
	The sequence $\{b_n\}_{n\geq 1}$ is dense in $[m,M]$, where $m=\inf \{b_n:n\geq 1\}, M=\sup \{b_n:n\geq 1\}$. That is to say, $
\{b_n:n\geq 1\}'=[m,M]$.

\end{thm}

\vspace{10pt}

Theorem \ref{thm dense} urges us  to discuss the ``homogeneity" of the distribution of the sequence $\{b_n\}_{n\geq 1}$.

In the process, we introduce the  function $\lambda: \mathbb{R}^+ \to \mathbb{R}^+$ with
\[\lambda(x)=\lim_{k\to\infty}\frac{a(s^kx)}{(s^kx)^{\log_s p}},\]
where $a(x):=a_{[x]}$. We show that $\lambda(sx)=\lambda(x),   \lambda(x)\in[m,M]$ for any $x>0$, and
\[\{\lambda(x):x\geq 0 \text{ and } x \text{ is } s\text{-ary irrational number}\}\]
is dense in $[m,M]$.
We also discuss the continuity and the level sets of the function $\lambda $  and obtain the general results,
show that $\lambda$ is always continuous from the right at any  $x>0$ and continuous from the left at $s$-ary irrational number $x>0$,
as well as $\{x>0:~~\lambda(x)=\alpha\}$ has measure zero for any $\alpha\in[m,M]$.

Based on this, we show that the sequence $\{b_n\}_{n\geq 1}$ is not uniformly distributed modulo 1 and it does not have the natural distribution function.
\begin{thm}\label{thm u.d.}
The sequence $\{b_{n}\}_{n\geq 1}$ is not uniformly distributed modulo 1.
\end{thm}
Here a real sequence $\{x_n\}_{n\geq 1}$ is said to be uniformly distributed modulo 1 (abbreviated u. d. mod 1)  if for any interval $I\subset[0,1]$ one has
\[\lim_{N\to\infty}\frac{\sharp\{1\leq n\leq N: \{x_n\}\in I\}}{N}=|I|,\]
where $|I|$ is the length of the interval $I$. The formal definition of u. d. mod 1 was given by Weyl (\cite{Weyl14,Weyl16}).

\begin{thm}\label{thm cum}
	The cumulative distribution function of the sequence $\{b_{n}\}_{n\geq 1}$ does not exist at any point $\alpha\in (m,M)$.
\end{thm}
By this we mean the limit $\lim\limits_{x\to\infty}x^{-1}D(x,\alpha)$ does not exist and $D(x,\alpha)$ denotes the number of times $b_n\leq \alpha$ for $1\leq n\leq x$.

\vspace{10pt}

In the positive direction, we prove that a modified distribution function for $\{b_n\}_{n\geq 1}$ does exist.
\begin{thm} \label{thm log}
	For any  $\alpha\in[m,M]$, the logarithmic distribution function of the sequence $\{b_{n}\}_{n\geq 1}$ exists at $\alpha$, and has the value
\[
L(\alpha)=\frac{1}{\ln s}\int_{E_\alpha}\frac{1}{x}dx,
\]
where $E_\alpha=\{x\in[s^{-1},1),\lambda(x)\le\alpha\}$ and the integral is a Lebesgue integral.
\end{thm}

The definition of the logarithmic distribution function used here comes from  \cite{Brillhart83}, which are defined as follows,
\[L(\alpha):=\lim_{x\to\infty}\frac{1}{\ln x}\sum_{1\leq n\leq x,b_n\leq \alpha}\frac{1}{n}.\]

\vspace{10pt}

At last, we consider the Cantor integers sequence consists of the non-negative integers  $n$ whose digits in the $p$-ary expansions  of them have the same remainder $r$  modulus  $q$, where positive integers $q \geq 2$ and $p>q+r$ with $r\in\{0,1,\ldots,q-1\}$. It can be seen as the special case of $h(x)=qx+r$. This is  why they are called the ``linear" Cantor integers sequences. We give the exact value of the supremum and infimum of $\{b_n\}_{n\geq 1}$.

\begin{thm}\label{thm bounded}
For any positive integers $q \geq 2$ and $p>q+r$ with $r\in\{0,1,\ldots,q-1\}$, let $s=\lceil \frac{p-r}{q}\rceil$ and $\lceil \cdot \rceil$ be the ceil function. When $A=\{qi+r: i\in\{0,1,\ldots,s-1\}\}$, one has
$$m=\frac{q (s-1)+r}{p-1},~~~M=\frac{q (p-1)+pr}{p-1}.$$
\end{thm}

\section{The connection of $\{b_n\}_{n\geq 1}$ and the self-similar measure $\mu_\mathfrak{C}$ }\label{sec measure}
Recall that the Cantor set
\[
\mathfrak{C}:=\left\{\sum_{n=1}^{\infty}\frac{\varepsilon_n}{p^n}: \varepsilon_n \in A \text{ for any positive integer } n\right\}
\]
is the attractor of the family of contracting self-maps $\{S_i\}_{i=0}^{s-1}$ of $[0,1]$ with
\[S_i(x)=\frac{x}{p}+\frac{h(i)}{p}, ~~~i\in\{0,1,\ldots,s-1\}.\]

Let $\textbf{P}(\mathfrak{C})$ be the set of the Borel probability measures on $\mathfrak{C}$, and
\[L(\nu,\nu')=\sup_{\mathrm{Lip}(g)\leq 1}\left|\int gd\nu-\int gd\nu'\right|, \text{ with } \mathrm{Lip} (g)=\sup_{x\neq y}\left|\frac{g(x)-g(y)}{x-y}\right|,\]
be the dual Lipschitz metric on the space  $\textbf{P}(\mathfrak{C})$. Since the mapping $F$ defined on $\textbf{P}(\mathfrak{C})$ by
\[F(\nu)=\sum_{i=0}^{s-1}\frac{1}{s}\nu S_i^{-1}\]
is a contracting self-map of the compact metric space $(\textbf{P}(\mathfrak{C}),L(\nu,\nu'))$. By the fundamental theorem on iterated function systems and the compact fixed-point theorem, one has
\[F^k(\delta_0)\to\mu_\mathfrak{C} \text{ as } k\to\infty,\]
where $\delta_0$ is the Dirac measure concentrated at $0$. Note that for any $x\in[0,1]$, $\mu_\mathfrak{C}(\{x\})=0$, one has
\[\lim_{k\to\infty}F^k(\delta_0)([0,x])=\mu_\mathfrak{C}([0,x]),\]
 by Portmanteau Theorem. Now, let us prove Theorem \ref{thm measure}.

For any $x\in\mathfrak{C}\cap\left [\frac{h(1)}{p},1\right]$, suppose
\[x=\sum_{i=1}^\infty\frac{h(\varepsilon_i)}{p^i}, \text{ where } \varepsilon_i\in\{0,1,\ldots,s-1\},\qquad  \varepsilon_1\neq 0.\]
Since
\[F^k(\delta_0)=\sum_{i_1\ldots i_k\in \{0,1,\ldots,s-1\}^k}\frac{1}{s^k}\delta_{S_{i_1}\circ  \cdots  \circ S_{i_k}(0)}.\]
Then for any $k\geq 1$, if we put $n_k=\sum_{i=1}^k\varepsilon_i s^{k-i}$, one has
\[a_{n_k}=[p^kx],\qquad x=\lim_{k\to\infty}\frac{a_{n_k}}{p^k},\]
and
\[ F^k(\delta_0)([0,x])=\frac{\sharp\{i_1\ldots i_k\in\{0,1,\ldots,s-1\}^k:S_{i_1}\circ \cdots \circ S_{i_k}(0)\in[0,x]\}}{s^k}=\frac{n_k+1}{s^k}.\]
Thus
\[\frac{x}{(\mu_\mathfrak{C}([0,x]))^{\log_s p}}=\lim_{k\to\infty}\frac{a_{n_k}/p^k}{((n_k+1)/s^k)^{\log_s p}}
=\lim_{k\to\infty}\frac{a_{n_k}}{(n_k+1)^{\log_s p}}=\lim_{k\to\infty}\frac{a_{n_k}}{n_k^{\log_s p}}.\]
That is to say, $\frac{x}{(\mu_\mathfrak{C}([0,x]))^{\log_s p}}$ is the limit point of $\{b_n\}_{n\geq 1}$.
\vspace{10pt}

For any limit point $\gamma$ of $\{b_n\}_{n\geq 1}$. Assume that $\{n_k\}_{k\geq 1}$ be a subsequence with
\[\lim_{k\to\infty}b_{n_k}=\gamma.\]
Note that $\left\{\frac{a_{n_k}}{p^{\ell_k}}\right\}_{k\geq 1}$ and $\left\{\frac{n_k}{s^{\ell_k}}\right\}_{k\geq 1}$ are both bounded, without loss of generation, we can further ask that
\[\lim_{k\to\infty}\frac{a_{n_k}}{p^{\ell_k}} ~~~~~~~~~~~~\text{and}~~~~~~~~~~~ \lim_{k\to \infty}\frac{n_k}{s^{\ell_k}}\]
both exist and denote the limit values are $x$ and $t$ respectively. It is clearly  that $x\in[h(1)/p,1]$ and $t\in[1/s,1]$.

For any $k\geq 1$, let $\ell_k$ be the unique integer such that $s^{\ell_k-1}\leq n_k <s^{\ell_k}$. At this time,
\[\gamma=\lim_{k\to\infty}b_{n_k}=\lim_{k\to\infty}\frac{a_{n_k}/p^{\ell_k}}{(n_k/s^{\ell_k})^{\log_s p}}
=\lim_{k\to\infty}\frac{a_{n_k}/p^{\ell_k}}{((n_k+1)/s^{\ell_k})^{\log_s p}}=\frac{x}{(\mu_\mathfrak{C}([0,x]))^{\log_s p}}\]
for the above $x$, since
\[
\frac{n_k+1}{s^{\ell_k}}=\left(F^{\ell_k}(\delta_0)\left(\left[0,\frac{a_{n_k}}{p^{\ell_k}}\right]\right)-F^{\ell_k}(\delta_0)([0,x])\right)
+\left(F^{\ell_k}(\delta_0)([0,x])-\mu_{\mathfrak{C}}([0,x])\right)+\mu_{\mathfrak{C}}([0,x]),
\]
and
\[\lim_{k\to\infty}\left(F^{\ell_k}(\delta_0)\left(\left[0,\frac{a_{n_k}}{p^{\ell_k}}\right]\right)-F^{\ell_k}(\delta_0)([0,x])\right)=0,
\qquad
\lim_{k\to\infty}F^{\ell_k}(\delta_0)([0,x])=\mu_{\mathfrak{C}}([0,x]).
\]

When $0\in A$, i. e., $h(0)=0$,
for any $x\in\mathfrak{C}\cap\left (0,\frac{h(1)}{p}\right)$, suppose
\[x=\sum_{i=1}^\infty\frac{h(\varepsilon_i)}{p^i}, \text{ where } \varepsilon_i\in\{0,1,\ldots,s-1\}, \text{ and } ~~k_0=\min\{i:\varepsilon_i\neq 0\}.\]
It is clear that $k_0>1$ and $p^{k_0-1}x\in \mathfrak{C}\cap\left [\frac{h(1)}{p},1\right]$. By the definition of $\mu_\mathfrak{C}$, one has $\mu_\mathfrak{C}([0,x])=s^{-k_0+1}\mu_\mathfrak{C}([0,p^{k_0-1}x])$, and thus
\[
\frac{x}{(\mu_\mathfrak{C}([0,x]))^{\log_s p}}=\frac{p^{k_0-1}x}{\mu_\mathfrak{C}([0,p^{k_0-1}x]))}.
\]
All of this means the results of Theorem \ref{thm measure} are correct.
\begin{rem}
Let $p=4, s=2, A=\{1,3\} (i.e.,h(x)=2x+1), x=\frac{3}{8}\in\mathfrak{C}\cap(\frac{1}{4},\frac{3}{4})=\mathfrak{C}\cap\left(\frac{h(0)}{p},\frac{h(1)}{p}\right)$. Then $\mu_\mathfrak{C}([0,x])=\frac{1}{4}$, and $\frac{x}{(\mu_\mathfrak{C}([0,x]))^{\log_2 4}}=6\notin[1,\frac{10}{3}]$. Where $m=1,M=\frac{10}{3}$ base on  Theorem \ref{thm bounded}.
\end{rem}

\section{The denseness of $\{b_n\}_{n\geq 1}$}\label{sec dense}
We begin the proof of Theorem \ref{thm dense} with the facts that
(1) $a_n\geq n$ for any integer $n\geq 1$,
(2) $a_i=h(i)$ for any $i\in\{0,1,\ldots, s-1\}$,
(3) $a_{sn+i}=pa_n+h(i)$ for any integer $n\geq 1$ and $i\in\{0,1,\ldots, s-1\}$,
as well as the following properties of $b_n$.

\begin{prop}\label{pro m}
For any positive integer $n$, one has $b_{s^\ell n+s^\ell-1}<b_n$ for any large $\ell$ with $1-s^{-\ell}>\log_p s$.
\end{prop}
\begin{proof}
Note that
\[b_{s^\ell n+s^\ell-1}
=\frac{p^\ell a_n+\frac{h(s-1)}{p-1}(p^\ell-1)}{(s^\ell n+s^\ell-1)^{\log_s p}}
\leq \frac{a_n+1}{(n+\log_p s)^{\log_s p}}
=b_n\cdot \frac{1+a_n^{-1}}{(1+n^{-1} \log_p s)^{\log_s p}},
\]
since $h(s-1)\leq p-1$ and $1-s^{-\ell}>\log_p s$. Combine this with the facts that
\[a_n\geq n~~ \text{   and    } ~~(1+n^{-1}\log_p s)^{\log_s p}> 1+n^{-1},\] one has $b_{s^\ell n+s^\ell-1}<b_n$.
\end{proof}

\begin{prop}\label{pro mon2}
$b_{s n+{s-1}}< b_{s n+{s-2}}< \cdots < b_{s n+1}< b_n$
when $n$ large enough, and $b_n \leq b_{s n}$ for any $n\geq 1$. Moreover $b_n=b_{s n}$ if and only if $h(0)=0$.
\end{prop}

\begin{proof}
Note that for any $i\in\{1,2,\ldots, s-2\}$, $b_{sn+i+1}<b_{sn+i}$ if and only if
\[1+\frac{h(i+1)-h(i)}{pa_n+h(i)}<\left(1+\frac{1}{sn+i}\right)^{\log_s p},\]
and the facts
\[\left(1+\frac{1}{sn+i}\right)^{\log_s p}>1+\log_s p\cdot \frac{1}{sn+i}, ~~~~~~~~~~~~0\leq h(i)\leq p-1.\]
We have $b_{s n+{s-1}}< b_{s n+{s-2}}< \cdots < b_{s n+1}$ for $n$ large enough
since the growth order of $a_n$ is $n^{\log_s p}$ and $\log_s p>1$.

Similarly,
\[1+\frac{h(1)}{pa_n}<\left(1+\frac{1}{sn}\right)^{\log_s p}\] implies $b_{sn+1}<b_n$ when $n$ large enough.

At last, for any positive integer $n$, $b_n\leq b_{sn}$ if and only if $h(0)\geq 0$, and $b_n=b_{sn}$ holds if and only if $h(0)=0$.
\end{proof}

For convenience, suppose the inequalities $b_{s n+{s-1}}< b_{s n+{s-2}}< \cdots < b_{s n+1}< b_n \leq b_{s n}$ in Proposition \ref{pro mon2} hold for $n>N_0$ and $s^{k_0-1}\leq N_0<s^{k_0}$ for some positive integer $k_0$.

At first, we show that infimum $m=\inf \{b_n:n\geq 1\}$ and the supremum $M=\sup \{b_n:n\geq 1\}$ are both the limit points of $\{b_{n}\}_{n\geq 1}$.

Note that $b_n\neq m$ for any integer $n\geq 1$, which follows from Proposition \ref{pro m}. Thus the infimum $m$ must be the limit point of  $\{b_{n}\}_{n\geq 1}$.

Similarly, the supremum $M$ must be the limit point of  $\{b_{n}\}_{n\geq 1}$ in the case of $h(0)\neq 0$ since $b_n<b_{sn}$ for any $n\geq 1$.

When $h(0)=0$. Let $b_{n_0}=\max\{b_n: 1\leq n< s^{k_0+1}\}$ for some $1\leq n_0< s^{k_0+1}$. Then for any $n\geq s^{k_0+1}$, assume $n=[\varepsilon_k \ldots \varepsilon_1 \varepsilon_0]_s$ with $k\geq k_0+1$, by Proposition \ref{pro mon2}, one has
\[b_n\leq b_{[\varepsilon_k \ldots \varepsilon_1]_s}\leq \cdots\leq b_{[\varepsilon_k \ldots \varepsilon_{k-k_0}]_s}\leq b_{n_0},\]
since $[\varepsilon_k \ldots \varepsilon_{k-k_0}]_s\geq s^{k_0}>N_0$. Thus
\[M=b_{n_0}=\lim_{k\to\infty}b_{s^kn_0}.\]
This is due to the fact that $b_{sn}=b_n$ for any $n\geq 1$ if $h(0)=0$.
Therefore, the supremum $M$ is also the limit point of $\{b_{n}\}_{n\geq 1}$ when $h(0)=0$.

On the basis of the above findings we offer the following Lemma which will be used in the construction of  the subsequence of $\{b_{n}\}_{n\geq 1}$  whose limit exactly belongs to $(m,M)$.

\begin{lem}
For any $\gamma\in(m,M)$, there exists an integer $n_1\geq s^{k_0}$ such that $b_{n_1+1}<\gamma\leq b_{n_1}$.
\end{lem}
\begin{proof}
Suppose the result does not hold, then for any $n\geq s^{k_0}$, $ b_{n+1}\geq \gamma$ or $b_{n}<\gamma$.
At this moment, for any $n>s^{k_0}$, $b_n=b_{(n-1)+1}<\gamma$ implies $b_{n-1}<\gamma$,
and $b_n\geq\gamma$ implies $b_{n+1}\geq\gamma$. Thus for any $n\geq s^{k_0}$,
\begin{itemize}
  \item If $b_n<\gamma$, then for any integer $m\in[s^{k_0},n]$ one has $b_m<\gamma$.
  \item If $b_n\geq \gamma$, then for any integer $m\geq n$ one has $b_m\geq \gamma$.
\end{itemize}
Put \[m_0:=\min\{n\geq s^{k_0}: b_n\geq \gamma\}.\]
It is well defined since $\gamma<M$ and $M$ is the limit point of  $\{b_{n}\}_{n\geq 1}$  imply $\{n\geq s^{k_0}: b_n\geq \gamma\}\neq\emptyset$.

Therefore, for any $n\geq m_0$, one has $b_n\geq \gamma$, which is contradict with $\gamma>m$ and $m$
is also the limit point of  $\{b_{n}\}_{n\geq 1}$.
\end{proof}

Now, we will show that for any point $\gamma\in(m,M)$, it is the limit point of  $\{b_{n}\}_{n\geq 1}$. That is to say, there is a subsequence $\{n_k\}_{k\geq 1}$ such that
\[\lim_{k\to\infty} b_{n_k}=\gamma.\]

Take $n_1\geq s^{k_0}$ with $b_{n_1+1}<\gamma\leq b_{n_1}$. And define $n_{k+1}$ recursively as follows.

\[
n_{k+1}=\left\{
  \begin{array}{ll}
    sn_k+i, & \hbox{if~~ $b_{sn_k+(i+1)}\textless\gamma\le b_{sn_k+i}$ ~~for some~~ $i\in \{0,1,\ldots,s-2\}$;} \\
    sn_k+(s-1), & \hbox{if~~ $b_{sn_k+(s-1)}\ge\gamma$.}
  \end{array}
\right.\]

At first, we will prove that the limit $\lim\limits_{k\to\infty}b_{n_k}$  exists. To do that, we just need to show that $ \sum\limits_{k=1}^{\infty}(b_{n_k}-b_{n_{k+1}})$ is convergent.  It is sufficient to show that  $\sum\limits_{k=1}^{\infty}|b_{n_k}-b_{n_{k+1}}|<\infty$.
		
Suppose $n_{k+1}=sn_k+i$ for some $i\in \{0,1,\ldots,s-1\}$.

If $i\in \{1,\ldots,s-1\}$, by Proposition \ref{pro mon2}, we have
\[
|b_{n_k}-b_{n_{k+1}}|=b_{n_k}-b_{n_{k+1}}
=\frac{a_{n_k}}{{n_k}^{\log_s p}}-\frac{p a_{n_k}+ h(i)}{(s n_k+ i)^{\log_s p}}
=\frac{a_{n_k}}{{n_k}^{\log_s p}}\cdot \left(1-\frac{1+\frac{h(i)}{pa_{n_k}}}{(1+\frac{i}{sn_k})^{\log_s p}}\right).
\]	
Note that
\[1-\frac{1+\frac{h(i)}{pa_{n_k}}}{(1+\frac{i}{sn_k})^{\log_s p}}
<\left(1+\frac{i}{sn_k}\right)^{\log_s p}-1-\frac{h(i)}{pa_{n_k}}
<\left(1+\frac{i}{sn_k}\right)^{\log_s p}-1.\]		
And
\[\left(1+\frac{i}{sn_k}\right)^{\log_s p}-1
=\log_s p \cdot \left(1+\frac{\theta i}{sn_k}\right)^{\log_s p-1}\cdot \frac{i}{sn_k}
<\log_s p\cdot \frac{p}{s}\cdot \frac{1}{n_k}\leq \frac{p\log_s p}{s^k},\]
since $\theta\in(0,1)$ implies $1+\frac{\theta i}{sn_k}<s$, and $n_k\ge s^{k-1}$.
Thus at this case,
\[
|b_{n_k}-b_{n_{k+1}}|<M\cdot\frac{p\log_s p}{s^k}=:Cs^{-k},
\]
where $C:=M p\log_s p$.

If $i=0$, by Proposition \ref{pro mon2}, we also have
\[
|b_{n_k}-b_{n_{k+1}}|=b_{n_{k+1}}-b_{n_k}=\frac{h(0)}{p{n_k}^{\log_s p}}\leq \frac{h(0)}{ps^{(k-1)\log_s p}}= \frac{h(0)}{p^k}<Cs^{-k},
\]
since $p>s$, $0\leq h(0)<p$ and $M\geq b_1=\frac{a_1}{1^{\log_s p}}=h(1)\geq 1$.
Therefore,
\[
\sum_{k=1}^{\infty}|b_{n_k}-b_{n_{k+1}}|<\sum_{k=1}^{\infty}Cs^{-k}=\frac{C}{s-1}<\infty.
\]

It is obviously that  $\lim\limits_{k\to\infty}b_{n_k}\ge\gamma$. It can be seen, the set $\{k\geq 1:n_k\not\equiv s-1(\text{mod  s} )\}$ is an infinite set.

If not, put $K=\max\{k\geq 1:n_k\not\equiv s-1(\text{mod  s} )\}$ if $\{k\geq 1:n_k\not\equiv s-1(\text{mod  s} )\}\neq \emptyset$ and $1$ otherwise. Then $b_{n_K+1}<\gamma\leq b_{n_K}$, and $n_k\equiv s-1(\text{mod  s} )$ for any $k>K$. Thus
  \[
  \lim_{k\to\infty}b_{n_k}
  =\lim_{\ell\to\infty}\frac{p^\ell a_{n_K}+h(s-1)\frac{p^\ell -1}{p-1}}{(s^\ell n_K+s^\ell-1)^{\log_s p}}
  =\frac{a_{n_K}+\frac{h(s-1)}{p-1}}{(n_K+1)^{\log_s p}}
  \leq \frac{a_{n_K}+1}{(n_K+1)^{\log_s p}}
  \leq b_{n_K+1}<\gamma,
  \]
which contradicts with $\lim\limits_{k\to\infty}b_{n_k}\ge\gamma$.

Now, we show that $\lim\limits_{k\to\infty}b_{n_k}\le\gamma$. Note that $n_k\not\equiv s-1(\text{mod  s} )$ for infinitely many $k\geq 1$.
One has $b_{sn_k+s-1}=\frac{pa_{n_k}+h(s-1)}{(sn_k+s-1)^{\log_s p}}<\gamma$ for infinitely many $k$.
We denote the set of such $k$ by $\mathcal{K}$, thus
  \[\lim_{k\to \infty}b_{n_k}
  =\lim_{k\in\mathcal{K},k\to \infty}\frac{a_{n_k}}{n_k^{\log_s p}}
  =\lim_{k\in\mathcal{K},k\to \infty}\frac{pa_{n_k}+h(s-1)}{(sn_k+s-1)^{\log_s p}}
  \leq \gamma,\]
  since $a_{n_k}\to\infty, n_k\to\infty$ as $k\to\infty$.	
In conclusion, we have $\lim\limits_{k\to\infty}b_{n_k}=\gamma$.

\section{The distribution of $\{b_n\}_{n\geq 1}$}\label{Sec distribution}
In this section, we shall come to the distribution of the sequence $\{b_n\}_{n\geq 1}$. For better research,  let us start by expanding the domain of the arithmetical function $b_n$ to positive real number. Just like what we have done in the proof of the dense of $\{b_n\}_{n\geq 1}$, if the subsequence $\{n_k\}_{k\geq 1}$ satisfies $n_{k+1}=sn_k+i$ for any $k\geq 1$ and some $i\in\{0,1,\ldots,s-1\}$, then $\lim\limits_{k\to\infty}b_{n_k}$ exists. So we can define the function
\[\lambda(x):=\lim_{k\to\infty}\frac{a(s^kx)}{(s^kx)^{\log_s p}}\]
for any $x\in(0,\infty)$,  where $a(x):=a_{[x]}$. It is clear that $a(x)=0$ for any $x\in(0,1)$. At first, we will prove strictly that $\lambda(x)$ is well-defined, rewrite the express of $\lambda(x)$ and estimate the convergence speed of the limit.

\begin{prop}\label{pro lam}
For any  $x>0$, the limit $\lim\limits_{k\to \infty}\frac{a(s^kx)}{(s^kx)^{\log_s p}}$ exists, $\lambda(sx)=\lambda(x)$, and $\lambda(x)\in[m,M]$.
Furthermore, the set $\{\lambda(x): x\in[s^{-1},1)\}$ is dense in $[m,M]$.
\end{prop}

\begin{proof}
For any $x>0$, suppose $x=[x]+[0.d_1 d_2 \ldots]_s$. Write $n_k=[s^k x]$. It is clear that $n_{k+1}=sn_k+d_k$ for any $k\geq 1$.
Thus,
\begin{equation}\label{4}
\lim_{k\to \infty}\frac{a(s^kx)}{(s^kx)^{\log_s p}}
=\lim_{k\to \infty}\frac{a_{n_k}}{n_k^{\log_s p}}\cdot\left(\frac{[s^kx]}{s^kx}\right)^{\log_s p}
=\lim_{k\to \infty}b_{n_k}.
\end{equation}
That is to say, the function $\lambda$ is well defined. The self-similarity of $\lambda$ can be obtained by the definition of itself. The boundedness of $\lambda$ and the denseness of $\{\lambda(x): x\in[s^{-1},1)\}$ follow equality (\ref{4}), the self-similarity of $\lambda$ and Theorem \ref{thm dense}.
\end{proof}

\begin{prop}\label{pro rewrite}
For $x\geq s^{-1}$, suppose $x=[x]+[0.d_1 d_2 \cdots]_s$ and let $\phi(x)=\sum_{j=1}^{\infty}{h(d_j)p^{-j}}$. Then
\[
\lambda(x)=\frac{a(x)+\phi(x)}{x^{\log_s p}}.
\]
And
\[
\left|\lambda(x)-\frac{a(s^kx)}{(s^k x)^{\log_s p}}\right|\le p^{-k}x^{-\log_s p}.
\]
for any $k\geq 1$.
\end{prop}

\begin{proof}
For $x\geq s^{-1}$ and $k\geq 1$, let $n_k:=[s^k x]$. Then $n_k=[x]s^k+\sum_{j=1}^k{d_j s^{k-j}}$, and
\[a(s^kx)=a_{n_k}=p^ka(x)+\sum_{j=1}^k{h(d_j)p^{k-j}}.\]
By the equality (\ref{4}),   we have
\[\lambda(x)
=\lim_{k\to \infty}\frac{p^ka(x)+\sum_{j=1}^k{h(d_j)p^{k-j}}}{(s^k x)^{\log_s p}}
=\lim_{k\to \infty}\frac{a(x)+\sum_{j=1}^k{h(d_j)p^{-j}}}{x^{\log_s p}}
=\frac{a(x)+\phi(x)}{x^{\log_s p}}.\]
It is clearly that $\left|\lambda(x)-\frac{a(x)}{x^{\log_s p}}\right|=\frac{\phi(x)}{x^{\log_s p}}\leq \frac{1}{x^{\log_s p}}$. Note that $\lambda(sx)=\lambda(x)$, then
\[\left|\lambda(x)-\frac{a(s^kx)}{(s^kx)^{\log_s p}}\right|=\left|\lambda(s^kx)-\frac{a(s^kx)}{(s^kx)^{\log_s p}}\right|
\le\frac{1}{(s^kx)^{\log_s p}}=p^{-k}x^{-\log_s p}.\]
\end{proof}

Now, let us  further explore the continuity of $\lambda(x)$.

\begin{prop}\label{prop_con}
The function $\lambda$ is always continuous from the right at any  $x>0$ and continuous from the left at $s$-ary irrational number $x>0$. $\lambda$ is continuous from the left at the $s$-ary rational number $x=[0.d_1\ldots d_N]$ with $d_N\in \{1,\ldots, s-1\}$ for some positive integer $N$ if and only if $h(s-1)= p-1$ and $h(d_N)-h(d_N-1)=1$.

Especially, if $h(x)=x+r$ for some positive integer $r$, and $s+r=p$, that is to say, $A=\{r,r+1,\ldots,p-1\}$, then the corresponding function $\lambda$ is continuous at any $x>0$.
\end{prop}

\begin{proof}
By the self-similarity of $\lambda$, we only need to consider $x\in [s^{-1},1)$.
Note that
\[\lambda(x)=\frac{\phi(x)}{x^{\log_s p}} ~\text{for any}~  x\in[s^{-1},1).\]
It is sufficient to consider the continuity of $\phi(x)=\sum\limits_{j=1}^{\infty}{h(d_j)p^{-j}}$ for $x= [0.d_1 d_2 \ldots]_s$ with $d_1\geq 1$.

For any $x\in[s^{-1}, 1)$, assume that $x= [0.d_1 d_2 \ldots]_s$ with $d_1\geq 1$. For any $n\geq 1$, take
\[x_n=[0.d_1\ldots d_n]_s+\frac{1}{s^n}.\]
 It is clear that $x<x_n$ and $\lim\limits_{n\to \infty}x_n=x$.
 And for any $x^*\in(x,x_n)$, suppose $x^*=[0. d_1^* d_2^*\ldots]_s$, then $d_k^*=d_k$ for any $1\leq k \leq n$. Thus,
\[|\phi(x^*)-\phi(x)|=\left|\sum\limits_{k=n+1}^{\infty}{h(d_k^*)p^{-k}}-\sum\limits_{k=n+1}^{\infty}{h(d_k)p^{-k}}\right|\leq \frac{1}{p^n}.\]
Therefore, $\lambda$ is continuous from the right at any  $x\in[s^{-1},1)$.

For any $s$-ary irrational number $x\in[s^{-1},1)$, suppose $x=[0.d_1d_2\ldots]_s$ with $d_1\geq 1$.
Set $x_n:=[0.d_1\ldots d_n]_s$. For any $x^*\in(x_n, x)$ with $x^*=[0. d_1^* d_2^*\ldots]_s$, one has $d_k^*=d_k$ for any $1\leq k \leq n$.
Similar to the discussion above, $|\phi(x^*)-\phi(x)|\leq \frac{1}{p^n}.$
Therefore, $\lambda$ is continuous from the left at any $s$-ary irrational $x\in[s^{-1},1)$.

For any $s$-ary  rational number $x\in[s^{-1},1)$. Suppose $x=[0.d_1\ldots d_N]_s$ for some  positive integer $N$ with $d_1, d_N\neq 0$.
For any integer $n>N$,  set
\[x_n:=x-s^{-n}=[0.d_1\ldots d_{N-1}(d_N-1)(s-1)^{n-N}]_s.\]
Then $x_n\to x$ as $n\to \infty$, and
\[
|\phi(x_n)-\phi(x)|=\phi(x)-\phi(x_n)
=\frac{h(d_N)-h(d_N-1)}{p^N}-\frac{h(s-1)}{p-1}\left(\frac{1}{p^N}-\frac{1}{p^n}\right).
\]
Note that\\
 If $h(s-1)\neq p-1$, then $|\phi(x_n)-\phi(x)|>p^{-(N+1)}$.\\
 If $h(s-1)= p-1, h(d_N)-h(d_N-1)\neq 1$, then $|\phi(x_n)-\phi(x)|>p^{-N}$.\\
 If $h(s-1)= p-1, h(d_N)-h(d_N-1)=1$, then $|\phi(x_n)-\phi(x)|=p^{-n}$, and thus for any $x^*\in (x_n,x)$, $|\phi(x^*)-\phi(x)|<p^{-n}$.

Therefore, the results hold.
\end{proof}

\begin{prop} \label{prop lam}
The set $\{\lambda(x): x\in [s^{-1},1) ~\text{and}~ x ~\text{is}~ s\text{-ary irrational number}\}$ is dense in $[m,M]$.
Hence, the set $\{\lambda(x): x\in [s^{-1},1) ~\text{and}~ \lambda ~\text{is continuous at}~ x\}$ is dense in $[m,M]$.
\end{prop}

\begin{proof}
For any $\gamma \in [m,M]$ and any $\delta>0$.
By Theorem \ref{thm dense}, there exists $n_k$ such that
\[\frac{M p^2}{n_k}<\delta \qquad \text{and} \qquad \left|\frac{a_{n_k}}{n_k^{\log_s p}}-\gamma\right|<\frac{\delta}{2}.\]
Suppose the $s$-ary expansion of $n_k$ is $n_k=[\varepsilon_{\ell_k} \ldots \varepsilon_1 \varepsilon_0]_s$. Take
$x_k=[0.\varepsilon_{\ell_k} \ldots \varepsilon_1 \varepsilon_0 (01)^\infty]_s$.
Then $x_k$ is $s$-ary irrational number which belongs to $[s^{-1},1)$. Note that
\[\frac{a_{n_k}}{n_k^{\log_s p}}=\frac{a(s^{\ell_k+1}x_k)}{(s^{\ell_k+1}x_k)^{\log_s p}}\cdot\left(1+\log_s p\cdot \left(1+\frac{\frac{\theta}{s^2-1}}{n_k}\right)^{\log_s p-1}\frac{\frac{1}{s^2-1}}{n_k}\right),\]
for some $\theta\in(0,1)$. Then
\[\left|\lambda(x_k)-\frac{a_{n_k}}{n_k^{\log_s p}}\right|\leq\left|\lambda(x_k)-\frac{a(s^{\ell_k+1}x_k)}{(s^{\ell_k+1}x_k)^{\log_s p}}\right|+\left|\log_s p\cdot \frac{a(s^{\ell_k+1}x_k)}{(s^{\ell_k+1}x_k)^{\log_s p}} \left(1+\frac{\frac{\theta}{s^2-1}}{n_k}\right)^{\log_s p-1}\frac{\frac{1}{s^2-1}}{n_k}\right|.\]
The first part of the right in the above inequality is no more than $p^{-\ell_k}$ by Proposition \ref{pro rewrite}.
And the second part is no more than
\[p\cdot M\cdot \frac{p}{s}\cdot \frac{1}{s^2-1}\cdot\frac{1}{n_k}<\frac{\delta}{6}.\]
Thus,
 \[|\lambda(x_k)-\gamma|<\frac{1}{p^{\ell_k}}+\frac{\delta}{6}+\frac{\delta}{2}\leq\delta,\]
 which implies the proposition is correct.

\end{proof}

\begin{lem}\label{lem bn}
For any interval $(e,f)\subset[m,M]$, there exist $x_0\in[1/s,1]$, $\eta_0\in(0,1)$ and $k_0\in \mathbb{Z}^+$ such that $b_n\in(e,f)$ for any integer $n\in[s^k(x_0-\eta_0),s^k(x_0+\eta_0)]$ with $k>k_0$.
\end{lem}

\begin{proof}
By Proposition \ref{prop lam}, there is  an $s$-ary irrational number $x_0\in(s^{-1},1)$ and real numbers $\widetilde{e}, \widetilde{f}$ such that $e<\widetilde{e}<\lambda(x_0)<\widetilde{f}<f$.
Since $\lambda$ is continuous at $x_0$, there exists $\eta_0\in(0,x_0-s^{-1})$ such that for any $x\in[x_0-\eta_0, x_0+\eta_0]\subset[s^{-1},+\infty)$, one has $\widetilde{e}<\lambda(x)<\widetilde{f}$.

Let $\delta=\min\{f-\widetilde{f},\widetilde{e}-e\}$, and choose $k_0$ so large that $p^{-k_0}(x_0-\eta_0)^{-\log_s p}<\delta$.
Thus for any $k\ge k_0$, by Proposition \ref{pro rewrite}, for any $x\in[x_0-\eta_0, x_0+\eta_0]$, one has
\[
\left|\lambda(x)-\frac{a(s^kx)}{(s^kx)^{\log_s p}}\right|\le p^{-k}x^{-\log_s p}\le p^{-k_0}(x_0-\eta_0)^{-\log_s p}\textless\delta.
\]
Then for any integer $n\in [s^k(x_0-\eta_0), s^k(x_0+\eta_0)]$,
\[
b_n=\frac{a_n}{n^{\log_s p}}
\leq\left|\frac{a_n}{n^{\log_s p}}-\lambda\left(\frac{n}{s^k}\right)\right|+\left|\lambda\left(\frac{n}{s^k}\right)
\right|<\delta+\widetilde{f}\leq f,
\]
and
\[
b_n=\frac{a_n}{n^{\log_s p}}
\geq\left|\lambda\left(\frac{n}{s^k}\right)
\right| - \left|\frac{a_n}{n^{\log_s p}}-\lambda\left(\frac{n}{s^k}\right)\right|>\widetilde{e}-\delta\geq e.
\]
\end{proof}

Now let us give \textbf{the proof of Theorem \ref{thm u.d.}}, show that $\{b_n\}_{n\geq 1}$ is not u. d. mod $1$.
Two situations are analyzed, calculated and discussed according to the position relation between $M$ and $m$.
\begin{proof}
\textbf{Case I: $[M]=[m]$ or $\lceil M\rceil=\lceil m\rceil$.}

Note that for any $n \geq 1$, $\{b_n\}=b_n-[m]\in[\{m\},\{M\}]$ and $[\{m\},\{M\}]\neq [0,1]$, thus $\{b_n\}_{n\geq 1}$ is not u. d. mod $1$.

\textbf{Case II: $[M]\neq[m]$ and $\lceil M\rceil\neq\lceil m\rceil$.}

Assume that the sequence $\{b_n\}_{n\geq 1}$ is u. d. mod $1$. Then for any $\alpha\in(0,1)$, one has
\[\lim_{N\to \infty}\frac{\sharp\{1\leq n\leq N: \{b_n\}\in[0,\alpha] \}}{N}=\alpha.\]
Let $\mathds{1}_E$ be the characteristic function of the set $E$, and $||x||$ be the distance from the real number $x$ to the integers, that is the infimum of $|x-n|$ over all $n\in\mathbb{Z}$. Put
\[\gamma:=\min\left\{\frac{1}{2}\mathds{1}_\mathbb{Z}(m)+||m||, \frac{1}{2}\mathds{1}_\mathbb{Z}(M)+||M||\right\}.\]
It is clear that $0<\gamma\leq \frac{1}{2}$.

For any $\alpha\in(0,\gamma)$, one has
\[ [\lceil m\rceil,\lceil m\rceil+\alpha]\subset[m,M],
\]
 By Lemma \ref{lem bn}, there exist $x_0\in[s^{-1},1), \eta_0\in(0,1), k_0\in\mathbb{Z}^+$ such that for any integer $n\in [s^k(x_0-\eta_0), s^k(x_0+\eta_0)]$ with $k\ge k_0$,  one has $b_n\in(\lceil m\rceil,\lceil m\rceil+\alpha)$, which implies that $\{b_n\}\in(0,\alpha)$.

Note that for any large enough $k$, the number of integers in interval $[s^k(x_0-\eta_0), s^k(x_0+\eta_0)]$ is
\[s^k(x_0+\eta_0)-s^k(x_0-\eta_0)=2s^k\eta_0+O(1).\] Thus
 \[
\sharp\{1\leq n\leq s^k(x_0+\eta_0): \{b_n\}\in[0,\alpha] \}=\sharp\{1\leq n\leq s^k(x_0-\eta_0): \{b_n\}\in[0,\alpha] \}+2 s^k\eta_0+O(1).
\]
 Dividing both sides by $s^k(x_0+\eta_0)$, and letting $k\to\infty$, then gives that
\[\alpha= \frac{x_0-\eta_0}{x_0+\eta_0} \alpha+\frac{2\eta_0}{x_0+\eta_0},\]
which implies $\alpha=1$, contradicts with $0<\alpha<\gamma\leq \frac{1}{2}$. Therefore, $\{b_n\}_{n\geq 1}$ is not u. d. mod $1$.
\end{proof}

Next, recall the definition of the cumulative distribution function in \cite{Brillhart83}.
\begin{dfn}\label{dfn cum}
Let $\{u_n\}_{n\geq 1}$ be a sequence of real numbers contained in an interval $I$. Let $\alpha \in I$  and let $D(x,\alpha)$ denote the number of $1\leq n \le x$ for which $u_n\le\alpha$, i.e.
\[D(x,\alpha)=\sum\limits_{1\le n\le x,u_n\le\alpha}1.\]
If the limit
\[\lim\limits_{x\to\infty}\frac{D(x,\alpha)}{x}=D(\alpha)\]
 exists, then the sequence $\{u_n\}_{n\geq 1}$ is said to have the distribution $D(\alpha)$ at $\alpha$. $D(\alpha)$ is called the cumulative distribution function of $\{u_n\}_{n\geq 1}$.
\end{dfn}

Let us give \textbf{the proof of Theorem \ref{thm cum}}. Show that the cumulative distribution function of the sequence $\{b_{n}\}_{n\geq 1}$ does not exist.

\begin{proof}
Assume that the cumulative distribution function of the sequence $\{b_{n}\}_{n\geq 1}$ exists at some $\alpha\in(m,M)$, denoted by
\[\lim_{x\to\infty}\frac{D(x,\alpha)}{x}=D(\alpha).\]

On the one hand, by Lemma \ref{lem bn}, there exist $x_1\in[s^{-1},1), \eta_1\in(0,1), k_1\in\mathbb{Z}^+$ such that $n\in [s^k(x_1-\eta_1), s^k(x_1+\eta_1)]$ implies that $b_n\in(m,\alpha)$ for any $k\ge k_1$.
Then for any large enough $k$,
\[ D(s^k(x_1+\eta_1),\alpha)=D(s^k(x_1-\eta_1),\alpha)+2 s^k\eta_1+O(1),\]
and thus $D(\alpha)=D(\alpha)\frac{x_1-\eta_1}{x_1+\eta_1}+\frac{2\eta_1}{x_1+\eta_1}$,
which implies $D(\alpha)=1$.

On the other hand, there exist $x_2\in[s^{-1},1), \eta_2\in(0,1), k_2\in\mathbb{Z}^+$ such that  $n\in [s^k(x_2-\eta_2), s^k(x_2+\eta_2)]$ implies that $b_n\in(\alpha,M)$ for any $k\ge k_2$.
Then for any large enough $k$,
\[D(s^k(x_2+\eta_2),\alpha)=D(s^k (x_2-\eta_2),\alpha),\]
and thus $D(\alpha)=D(\alpha)\frac{x_2-\eta_2}{x_2+\eta_2}$,
which implies $D(\alpha)=0$.

This is the contradiction.
And therefore $\lim\limits_{x\to\infty}\frac{D(x,\alpha)}{x}$ does not exist.
\end{proof}

In order to show that the logarithmic distribution function of the sequence $\{b_{n}\}_{n\geq 1}$ does exist, we need the properties of the level set of $\lambda(x)$.

\begin{prop}\label{prop measurezero}
For any $\alpha\in[m,M]$,  the set
\[S_\alpha=\{x>0:~~\lambda(x)=\alpha\}\] has measure zero.
\end{prop}

\begin{proof}
By the self-similarity of $\lambda$, we only need to consider the set of $x\in[s^{-1},1)$ with $\lambda(x)=\alpha$, which is still denoted by $S_\alpha$.
Let $\mathcal{N}:=\{x: x \text{ is normal to base } s\}$.

If $S_{\alpha}\cap \mathcal{N}=\emptyset$, then $S_{\alpha}$ has measure zero, since almost all real numbers (in the sense of Lebesgue measure) are absolutely normal \cite{Kuipers}. The result is correct.

If not, for any $x=[0. d_1 d_2 \ldots]_s\in S_{\alpha}\cap \mathcal{N}$,
there are infinitely many positive integers $n$'s such that $d_n \ldots d_{n+s-1}=0^s$.
That is to say, $\mathcal{I}:=\{n: d_n \ldots d_{n+s-1}=0^s\}$ is an infinite set.
For any  $n\in\mathcal{I}$, take
\[x_n=x+s^{-n}=[0.d_1\ldots d_{n-1}10^{s-1}d_{n+s}\ldots]_s,\]
\[y_n=x+s^{-n}+s^{-(n+s-1)}=[0.d_1\ldots d_{n-1}10^{s-2}1d_{n+s}\ldots]_s.\]
Then for any $x^* =[0.d_1^* d_2^* \ldots]_s\in (x_n, y_n)$, one has $d_1^* \cdots d_{n+s-2}^* =d_1\ldots d_{n-1}10^{s-2}$.
Thus
\begin{align*}
|\phi(x^*)-\phi(x)|&=\left|\frac{h(1)-h(0)}{p^n}+\sum_{k=n+s-1}^\infty\frac{h(d_k^*)-h(d_k)}{p^k}\right|\\
&\leq \frac{h(1)-h(0)}{p^n}+\sum_{k=n+s-1}^\infty\frac{h(s-1)}{p^k}\leq \frac{p-1}{p^n}+\frac{1}{p^{n+s-2}}
\leq \frac{1}{p^{n-1}}.
\end{align*}
It follows from  the above inequality that there exists $N>0$ such that for any $n>N$ with $n\in\mathcal{I}$,
\begin{align*}
|\lambda(x^*)-\lambda(x)|=&\left|\frac{\phi(x^*)}{{x^*}^{\log_s p}}-\frac{\phi(x)}{{x}^{\log_s p}}\right|
=\left|\frac{\phi(x^*)}{{x^*}^{\log_s p}}\left(\frac{{x^*}^{\log_s p}}{x^{\log_s p}}-1\right)-\frac{\phi(x^*)-\phi(x)}{{x}^{\log_s p}} \right|\\
\geq& m \cdot \log_s p \cdot \frac{x^*-x}{x}-\frac{1}{p^{n-1} }\cdot\frac{1}{ x^{\log_s p}}
\geq m s^{-n}-p^{-n+2} >0,
\end{align*}
since $x^*-x>s^{-n}$, $\log_s p>1$, and $x\in[s^{-1},1)$.
That is to say, for sufficiently large  $n\in\mathcal{I}$ and  any $ x^*\in(x_n, y_n)$, one has $\lambda(x^*)\neq\alpha$.
Therefore, when $n\in\mathcal{I}$ large enough,
\[
\frac{1}{s^{-n}+s^{-(n+s-1)}}\mathcal{L}(S_\alpha\cap (x,y_n))
=\frac{1}{s^{-n}+s^{-(n+s-1)}}\mathcal{L}(S_\alpha\cap (x,x_n))
\le\frac{s^{-n}}{s^{-n}+s^{-n-s+1}}=\frac{1}{1+s^{1-s}}<1,
\]
where $\mathcal{L}$ denote the Lebesgue measure.
But at the same time,
\[
\frac{1}{h}\mathcal{L}(S_\alpha\cap (x,x+h))=\frac{1}{h}\int_{x}^{x+h}\mathds{1}_{S_\alpha}(t)dt\to \mathds{1}_{S_\alpha}(x),
\]
as $h\to 0$ for almost all real number $x$, where $\mathds{1}_E(x)=1$ if and only if $x\in E$. Combining the above two conclusions, one has that the set $ S_\alpha\cap \mathcal{N}$ has measure zero.
Therefore
\[
L(S_\alpha)=L(S_\alpha\cap \mathcal{N})+L(S_\alpha\cap \mathcal{N}^c)=0.
\]
\end{proof}

\begin{lem}\label{lem_integral}
For any $\alpha\in[m,M]$,  let $E_\alpha=\{x\in[s^{-1},1): ~\lambda(x)\le\alpha\}$, then $\frac{\mathds{1}_{E_\alpha}(x)}{x}$ is Riemann integral.
\end{lem}
\begin{proof}
Since the conclusion is equivalent to that the points in $[s^{-1},1)$ at which $\frac{\mathds{1}_{E_\alpha}(x)}{x}$ fails to be continuous has measure zero.
It is sufficient to show that the set
\[\{x\in [s^{-1},1): ~  \mathds{1}_{E_\alpha}(x) \text{ is  not continuous at } x\}\]
has measure zero.

Let $x\in(s^{-1},1)$ be an $s$-ary irrational number with $\lambda(x)\neq \alpha$. By Proposition \ref{prop_con}, $\lambda$ is continuous at $x$. If $\lambda(x)> \alpha$, there exists $\delta_1>0$, such that for any $y\in(x-\delta_1, x+\delta_1)$, $\lambda(y)> \alpha$, which implies $\mathds{1}_{E_\alpha}(y)=\mathds{1}_{E_\alpha}(x)=0$.  Similarly,  If $\lambda(x)< \alpha$, there exists $\delta_2>0$, such that for any $y\in(x-\delta_2, x+\delta_2)$, $\mathds{1}_{E_\alpha}(y)=\mathds{1}_{E_\alpha}(x)=1$. Thus, $\mathds{1}_{E_\alpha}$ is continuous at $x$. Therefore,
\[\{x\in [s^{-1},1): ~ x \text{ is } s\text{-ary irrational number with } \lambda(x)\neq \alpha\}
\subset\{x\in [s^{-1},1): ~  \mathds{1}_{E_\alpha}(x) \text{ is  continuous at } x\}.\]
The result is true by Proposition \ref{prop measurezero}.
\end{proof}

Now, recall the definition of the logarithmic distribution.
\begin{dfn}\label{dfn log}
Let $\{u_n\}_{n\geq 1}$ be a real sequence contained in an interval $I$. Let $\alpha\in I$ and let
\[
L(x,\alpha)=\sum\limits_{1\le n\le x,u_n\le\alpha}\frac{1}{n}.
\]
If the limit
\[
\lim_{x\to \infty}\frac{1}{\ln x}L(x,\alpha)=L(\alpha)
\]
exists, then the sequence $\{u_n\}_{n\geq 1}$ is said to have the logarithmic distribution $L(\alpha)$ at $\alpha$. $L(\alpha)$ is called the logarithmic distribution function of the sequence $\{u_n\}_{n\geq 1}$.
\end{dfn}
And  complete \textbf{the proof of the Theorem \ref{thm log}}, show that the logarithmic distribution function does exist for the sequence $\{b_n\}_{n\geq 1}$.

\begin{proof}

For any positive integer $k$, let $I_k:=\{n\in\mathbb{Z}^+:s^{k-1}\le n< s^k\}$.
By Proposition \ref{pro rewrite}, for any $n\in I_k$,
\[|\lambda(n)-b_n|=|\lambda(n\cdot s^{-k})-b_n|\le p^{-k}(n\cdot s^{-k})^{-\log_s p}\leq p^{-(k-1)}.\]
Thus
\[
\{n\in I_k:\lambda(n)\leq \alpha-p^{-(k-1)}\}
\subseteq\{n\in I_k:b_n\leq \alpha\}
\subseteq\{n\in I_k:\lambda(n)\leq \alpha+p^{-(k-1)}\},
\]
which implies that
\[\sigma^*_k(\alpha-p^{-(k-1)})\le \sigma_k(\alpha)\le\sigma_k^*(\alpha+p^{-(k-1)}),\]
where
\[\sigma_k(\alpha):=\sum_{n\in I_k,b_n\le\alpha}\frac{1}{n}, \quad
\text{and} \quad \sigma_k^*(\alpha):=\sum_{n\in I_k,\lambda(n)\le\alpha}\frac{1}{n}.\]
We can rewrite $ \sigma_k^*(\alpha)$ as
\[
\sigma_k^*(\alpha)=\sum_{n\in I_k,\lambda(n\cdot s^{-k})\le\alpha}\frac{1}{n}
=\sum\limits_{n=s^{k-1}}^{s^{k}-1}\frac{\mathds{1}_{E_\alpha}(\frac{n}{s^k})}{\frac{n}{s^k}}\cdot \frac{1}{s^k}
.\]
By Lemma \ref{lem_integral},  $\frac{\mathds{1}_{E_\alpha}(x)}{x}$ is Riemann integral, we have
\[
\lim_{k\to\infty}\sigma_k^*(\alpha)=\int_{s^{-1}}^{1}\frac{\mathds{1}_{E_\alpha}(x)}{x}dx=\int_{E_\alpha}\frac{1}{x}dx=:h(\alpha).
\]

Note that
\[\bigcap_{m\geq 1}E_{\alpha+\frac{1}{m}}=E_\alpha, \quad \bigcup_{m\geq 1}E_{\alpha-\frac{1}{m}}=E_{\alpha}-S_{\alpha}.\]
From the continuity of the integration, and Proposition \ref{prop measurezero}, we can know that
\[
\lim_{m\to\infty}\int_{E_{\alpha+\frac{1}{m}}}\frac{1}{x}dx=\int_{E_\alpha}\frac{1}{x}dx
=\int_{E_\alpha-S_\alpha}\frac{1}{x}dx=\lim_{m\to\infty}\int_{E_{\alpha-\frac{1}{m}}}\frac{1}{x}dx.
\]
Therefore  $h(\alpha)$ is continuous, since $h(\alpha)$ is increasing with respect to $\alpha$.

For any fixed positive integer $k_0$, and any integer $k>k_0$,
\[
\sigma_k^*(\alpha-p^{-(k_0-1)})\le\sigma_k^*(\alpha-p^{-(k-1)})\le\sigma_k^*(\alpha+p^{-(k-1)})\le\sigma_k^*(\alpha+p^{-(k_0-1)}).
\]
First, let $k\to\infty$, we can get
\begin{align*}
h(\alpha-p^{-(k_0-1)})\le\lim_{k\to\infty}\inf\sigma_k^*(\alpha-p^{-(k-1)})\le\lim_{k\to\infty}\sup\sigma_k^*(\alpha-p^{-(k-1)})\le h(\alpha+p^{-(k_0-1)}),\\
h(\alpha-p^{-(k_0-1)})\le\lim_{k\to\infty}\inf\sigma_k^*(\alpha+p^{-(k-1)})\le\lim_{k\to\infty}\sup\sigma_k^*(\alpha+p^{-(k-1)})\le h(\alpha+p^{-(k_0-1)}).
\end{align*}
And then let $k_0\to\infty$, by the continuity of $h(\alpha)$, we have
\[
\lim_{k\to\infty}\sigma_k^*(\alpha-p^{-(k-1)})=\lim_{k\to\infty}\sigma_k^*(\alpha+p^{-(k-1)})=h(\alpha).
\]
Therefore,
\[\lim_{k\to\infty} \sigma_k(\alpha)=h(\alpha),\]
 which can in turn
\[
\lim_{m\to\infty}\frac{1}{m}\sum\limits_{k=1}^{m}\sigma_k(\alpha)=h(\alpha).
\]
Then for any $x>0$, choose integer $m$ such that $s^{m-1}\leq x<s^m$,
\[
\lim_{x\to\infty}\frac{1}{\ln x}\sum\limits_{1\leq n\leq x,b_n\le\alpha}\frac{1}{n}
=\lim_{m\to\infty}\frac{1}{m \ln s}\sum\limits_{k=1}^{m}\sigma_k(\alpha)
=\frac{1}{\ln s}h(\alpha).
\]
Thus, the logarithmic distribution function of the sequence $\{b_{n}\}_{n\geq 1}$ exists.
\end{proof}

\section{The supremum and infimum of $\{b_n\}_{n\geq 1}$ corresponding to linear Cantor integers}\label{sec bound}

We begin the proof of Theorem \ref{thm bounded} with the fact $a_n\geq \left(q+\frac{r}{s-1}\right)n$ and some properties of $b_n$.

\begin{prop}\label{pro b^r}
For any non-negative integer $k$ and any integer $n=[\varepsilon_k \varepsilon_{k-1} \ldots \varepsilon_0]_s\in[s^k,s^{k+1})$,
\begin{equation}
b_n=\widetilde{b}_n+r\cdot\frac{p^{k+1}-1}{p-1}\cdot n^{-\log_s p},
\end{equation}
where $\widetilde{b}_n=\frac{\widetilde{a}_n}{n^{log_s p}}$, and ~$\widetilde{a}_n:=[(q\varepsilon_k) (q\varepsilon_{k-1}) \ldots (q\varepsilon_0)]_p$.
\end{prop}

\begin{proof}
\[b_n=\frac{a_n}{n^{\log_s p}}
=\frac{\sum_{i=0}^k(q\varepsilon_i+r)p^i}{n^{\log_s p}}
=\frac{\widetilde{a}_n+r\frac{p^{k+1}-1}{p-1}}{n^{\log_s p}}
=\widetilde{b}_n+r\cdot\frac{p^{k+1}-1}{p-1}\cdot n^{-\log_s p}.\]
\end{proof}

\begin{prop}\label{pro mon}
Given a positive integer $k$. For any non-negative integer $\ell< k$ and any $\varepsilon_k \ldots \varepsilon_{\ell+1} \in \{0,1,\ldots,s-1\}^{k-\ell}$ with $\varepsilon_k\neq 0$, we have that $b_{[\varepsilon_k \ldots \varepsilon_{\ell+1}\varepsilon_\ell (s-1)^\ell]_s}$ decreases with the increasing of $\varepsilon_\ell$. That is to say,
\begin{equation}\label{equ mon}
b_{[\varepsilon_k \ldots \varepsilon_{\ell+1} (s-1)(s-1)^\ell]_s}<b_{[\varepsilon_k \ldots \varepsilon_{\ell+1} (s-2)(s-1)^\ell]_s}<\cdots <b_{[\varepsilon_k \ldots \varepsilon_{\ell+1} 1 (s-1)^\ell]_s}<b_{[\varepsilon_k \ldots \varepsilon_{\ell+1} 0 (s-1)^\ell]_s}.
\end{equation}
Especially, for any integer $n\geq 1$,
\[b_{s n+{s-1}}< b_{s n+{s-2}}< \cdots < b_{s n+1}< b_n \leq b_{s n}.\]
And $b_{s n}=b_n$ if and only if $r=0$.
\end{prop}

\begin{proof}
By Proposition \ref{pro b^r}, we only need to show that the inequality (\ref{equ mon}) holds for $\widetilde{b}_n$. That is to say
\begin{equation}\label{equ mon2}
\widetilde{b}_{[\varepsilon_k \ldots \varepsilon_{\ell+1} (s-1)(s-1)^\ell]_s}<\widetilde{b}_{[\varepsilon_k \ldots \varepsilon_{\ell+1} (s-2)(s-1)^\ell]_s}<\cdots <\widetilde{b}_{[\varepsilon_k \ldots \varepsilon_{\ell+1} 1 (s-1)^\ell]_s}<\widetilde{b}_{[\varepsilon_k \ldots \varepsilon_{\ell+1} 0 (s-1)^\ell]_s}.
\end{equation}

Fix $0\leq \ell< k$, write $n^*={[\varepsilon_k\varepsilon_{k-1}\ldots\varepsilon_{\ell+1}]}_s$. Then $n^*\geq 1$ and
$$\widetilde{b}_{[\varepsilon_k \ldots \varepsilon_{\ell+1}\varepsilon_\ell (s-1)^\ell]_s}=\frac{p^{\ell+1}\widetilde{a}_{n^*}+q\varepsilon_\ell p^\ell+\frac{q(s-1)(p^\ell-1)}{p-1}}{(s^{\ell+1}n^*+\varepsilon_\ell s^\ell+s^\ell-1)^{\log_s p}}.$$

For $x\geq 0$, let
\begin{equation*}
f(x)=\frac{p^{\ell+1}\widetilde{a}_{n^*}+q p^\ell x +\frac{q(s-1)(p^\ell-1)}{p-1}}{(s^{\ell+1}n^*+s^\ell x +s^\ell-1)^{\log_s p}}.
\end{equation*}
Now it suffices to show $f^\prime(x)<0$ for any $x>0$.

Through calculation and analysis, one has
\[
\text{sgn} (f^\prime(x))=\text{sgn}\left(q\cdot p^\ell-\log_s p\cdot\frac{p^{\ell+1}\widetilde{a}_{n^*}+q p^\ell x +\frac{q(s-1)(p^\ell-1)}{p-1}}{s n^*+ x +1-s^{-\ell}}\right).
\]

Since
\[\log_s p>1, ~~~~~~~~~~p^{\ell+1}\widetilde{a}_{n^*}+q p^\ell x +\frac{q(s-1)(p^\ell-1)}{p-1}> p^{\ell+1}qn^*+q p^\ell x,\]
and
\[s n^*+ x +1-s^{-\ell}<s n^*+ x +1,\]
implies that \[
q\cdot p^\ell-\log_s p\cdot\frac{p^{\ell+1}\widetilde{a}_{n^*}+q p^\ell x +\frac{q(s-1)(p^\ell-1)}{p-1}}{s n^*+ x +1-s^{-\ell}}<q\cdot p^\ell-\frac{p^{\ell+1}qn^*+q p^\ell x}{s n^*+ x +1}=-\frac{q p^\ell((p-s)n^*-1)}{sn^*+x-1}<0.
\]
Thus, $f^\prime(x)<0$ for any $x>0$. Therefore, the inequality (\ref{equ mon2}) holds.

Take $\ell=0$ in the inequality (\ref{equ mon}),  we have for any integer $n\geq 1$,
\[b_{s n+{s-1}}< b_{s n+{s-2}}< \cdots < b_{s n+1} < b_{s n}.\]
At this time, the inequality
\[b_{s n+1}=\frac{a_{s n+1}}{(s n+1)^{\log_s p}}=\frac{pa_{n}+q+r}{p n^{\log_s p}(1+(sn)^{-1})^{\log_s p}}<\frac{a_{n}}{n^{\log_s p}}=b_n\]
holds if and only if $pa_n\left((1+(sn)^{-1})^{\log_s p}-1\right)>q+r$.
Note that \[(1+(sn)^{-1})^{\log_s p}-1 >\log_s p\cdot (sn)^{-1} > (sn)^{-1},~~~ p>q(s-1), ~~~a_n\geq\left(q+\frac{r}{s-1}\right)n,\]
since $\log_s p>1$. It is sufficient to show that $q^2(s-1)+qr-(q+r)s\geq 0$, which follows from
\[q^2(s-1)+qr-(q+r)s=(q^2-q-r)s-q^2+qr\geq 2(q^2-q-r)-q^2+qr=(q-2)(q+r)\geq 0.\]
By combining this with Proposition \ref{pro mon2}, one has the results.
\end{proof}

It is important to notice that $b_{[\varepsilon_k (s-1)^k]_s}$ is not necessarily monotonous  with respect to $\varepsilon_k$. But it also get the minimal value at $\varepsilon_k=s-1$.

\begin{prop}\label{pro min}
For any $k\geq 0$, one has
\begin{equation}\label{equ min}
\min \left\{b_{[\varepsilon_k (s-1)^k]_s}: \varepsilon_k\in\{1,\ldots, s-1\}\right\}= b_{[(s-1) (s-1)^k]_s}.
\end{equation}

Especially, $b_{s-1}\leq b_{s-2}\leq \cdots \leq b_1$.
\end{prop}
\begin{proof}
By Proposition \ref{pro b^r}, we only need to show that the inequality (\ref{equ min}) holds for $\widetilde{b}_n$. That is to say
\[\widetilde{b}_{[\varepsilon_k (s-1)^k]_s}\geq \widetilde{b}_{[(s-1) (s-1)^k]_s} \text{ for any } \varepsilon_k\in\{1,\ldots, s-1\}.\]
Since $$\widetilde{b}_{[\varepsilon_k (s-1)^k]_s}=\frac{q\varepsilon_k p^k+\frac{q(s-1)(p^k-1)}{p-1}}{(\varepsilon_k s^k+s^k-1)^{\log_s p}}.$$
We consider the function
\begin{equation*}
f(x):=\frac{q p^k x+\frac{q(s-1)(p^k-1)}{p-1}}{(s^k x  +s^k-1)^{\log_s p}},
\end{equation*} which is defined for any $x\geq 1$.
Through simple calculation, we have
\[\text{sgn}(f^\prime(x))=\text{sgn}\left(p^k(s^k x+s^k-1)-\log_s p\cdot s^k\left(p^k x+\frac{(s-1)(p^k-1)}{p-1}\right)\right).\]
Note that
\[p^k(s^k x+s^k-1)-\log_s p\cdot s^k\left(p^k x+\frac{(s-1)(p^k-1)}{p-1}\right)\]
is a linear function with respect to $x$, and the coefficient of $x$ is $-p^ks^k(\log_s p-1)$, which is less than zero.
Thus, \[\min\{f(x):1\leq x \leq s-1\}=\min\{f(1),f(s-1)\}.\]
Now, we only need to show that for any $s\geq 3$, $f(1)\geq f(s-1)$. Namely
\begin{align}\label{1}
	\frac{q p^k+\frac{q (s-1)(p^k-1)}{p-1}}{(2s^k-1)^{\log_s p}}\ge\frac{q p^k(s-1)+
	\frac{q (s-1)(p^k-1)}{p-1}}{(s^{k+1}-1)^{\log_s p}}.
\end{align}
Let $u:=2s^k-1$, $v:=q\cdot p^k+\frac{q(s-1)(p^k-1)}{p-1}$. The right part of the inequality (\ref{1}) is converted to
\[
\frac{v\cdot \frac{(s-1)p^{k+1}-(s-1)}{(p+s-2)p^k-(s-1)}}{(\frac{s}{2}u+\frac{s-2}{2})^{\log_s p}}
\leq\frac{v\cdot \frac{(s-1)p^{k+1}-(s-1)}{(p+s-2)p^k-(s-1)}}{(\frac{s}{2}u)^{\log_s p}}
=\frac{v\cdot \frac{(s-1)p^{k+1}-(s-1)}{(p+s-2)p^k-(s-1)}}{ 2^{-\log_s p}\cdot p\cdot u{^{\log_s p}}}.
\]
We just need to show
\[ \frac{p}{2^{\log_s p}}\ge\frac{(s-1)p^{k+1}-(s-1)}{(p+s-2)p^k-(s-1)},\]
which is equivalent of
\begin{equation}\label{2}
(p+s-2-2^{\log_s p}(s-1))p^{k+1}-(p-2^{\log_s p})(s-1)\ge 0.
\end{equation}
Note that
\begin{equation}\label{3}
	2^{\log_s p}=p^{\log_s 2}=s^{\log_s 2}(1+\frac{p-s}{s})^{\log_s 2}<2(1+{\log_s 2}\cdot\frac{p-s}{s})<2(1+\frac{3}{7}\cdot\frac{p-s}{s-1}),
\end{equation}
since $\log_s 2<\frac{3}{7}\cdot\frac{s}{s-1}<1 $ when $s\geq 3$.  We can obtain that
\[p+s-2-2^{\log_s p}(s-1)> p+s-2-2(1+\frac{3}{7}\cdot\frac{p-s}{s-1})(s-1)=\frac{1}{7}(p-s)>0.\]
Thus
\[(p+s-2-2^{\log_s p}(s-1))p^{k+1}\geq(p+s-2-2^{\log_s p}(s-1))p^2,\]
and this causes the left part of the inequality (\ref{2}) is not less than
\[(p+s-2)p^2-2^{\log_s p}(s-1)(p^2-1)-p(s-1).\]
By inequality (\ref{3}), it can be further decreased to
\[(p+s-2)p^2-2(1+\frac{3}{7}\cdot\frac{p-s}{s-1})(s-1)(p^2-1)-p(s-1)=\frac{1}{7}(p-s)p^2+\frac{6}{7}(p-s)-(p-2)(s-1),
\]
and then is not less than
\[
\frac{1}{7}[(q-1)p^2+6(q-1)-7p+14](s-1)\ge\frac{1}{7}[p^2-7p+20](s-1)>0,
\]
for any $q\ge 2$, since $p-s\ge q(s-1)+1-s=(q-1)(s-1)$. That is to say, the inequality (\ref{2}) holds.
\end{proof}

\begin{cor}\label{cor bound}
For any non-negative integer $k$, and any integer $n\in [s^k, s^{k+1})$, $b_{s^{k+1}-1} \leq b_n\leq b_{s^k}$.
\end{cor}

\begin{proof}
For any integer $n\in [s^k, s^{k+1})$, suppose $n=[\varepsilon_k \varepsilon_{k-1} \ldots \varepsilon_1 \varepsilon_0]_s$.

At first, by Proposition \ref{pro mon} and Proposition \ref{pro min}, the value of $b_n$ decreases if we replace the last digit which does not equal to $s-1$ in the $s$-ary expansion of $n$ with $s-1$ from the heading of the expansion. Repeating this process, we can obtain,
\[b_n=b_{[\varepsilon_k \varepsilon_{k-1} \ldots \varepsilon_1 \varepsilon_0]_s}
\geq b_{[\varepsilon_k \varepsilon_{k-1} \ldots \varepsilon_1 (s-1)]_s}
\geq b_{[\varepsilon_k \varepsilon_{k-1} \ldots (s-1) (s-1)]_s}
\geq \cdots \geq  b_{[(s-1)^{k+1}]_s}=b_{s^{k+1}-1}.\]
Secondly, by inequality (\ref{equ mon2}) in the proof of Proposition \ref{pro mon} and the definition of $\widetilde{b}_n$, we have
\[\widetilde{b}_{sm+i}\leq \widetilde{b}_m=\widetilde{b}_{sm} \text{ for any integer }  m\geq 1 \text{ and } i\in\{0, 1, \ldots, s-1\}.\]
 Thus
\[\widetilde{b}_n=\widetilde{b}_{[\varepsilon_k \varepsilon_{k-1} \ldots \varepsilon_1 \varepsilon_0]_s}
\leq \widetilde{b}_{[\varepsilon_k \varepsilon_{k-1} \ldots \varepsilon_1 ]_s}
\leq \cdots \leq \widetilde{b}_{[\varepsilon_k]_s}\leq \widetilde{b}_{[1]_s}
=\widetilde{b}_{[10]_s}= \cdots =  \widetilde{b}_{[10^k]_s}=\widetilde{b}_{s^k}.\]
Combine this with Proposition \ref{pro b^r}, we have $b_n\leq b_{s^k}$.
\end{proof}

Now, let complete \textbf{the  proof of Theorem \ref{thm bounded}}.

Using Corollary \ref{cor bound}, it is sufficient to show that for any non-negative integer $k$,
\[\frac{q (s-1)+r}{p-1}\leq b_{s^{k+1}-1} \quad \text{ and } \quad b_{s^k}\leq \frac{q (p-1)+pr}{p-1}.\]
Since
\[
b_{s^k}=\frac{a_{s^k}}{(s^k)^{\log_s p}}=\frac{(q+r)p^k+\frac{r(p^k-1)}{p-1}}{p^k}=\frac{q (p-1)+pr}{p-1}-\frac{r}{(p-1)p^k},
\]
which monotonously increases with the increase of $k$, one has
\[b_{s^k}\leq \lim _{k\to\infty}b_{s^k}=\frac{q (p-1)+pr}{p-1}.
\]
Similarly,
\[
b_{s^{k+1}-1}=\frac{a_{s^{k+1}-1}}{(s^{k+1}-1)^{\log_s p}}=\frac{\frac{(q(s-1)+r)(p^{k+1}-1)}{p-1}}{(s^{k+1}-1)^{\log_s p}}=\frac{q(s-1)+r}{p-1}\cdot \frac{p^{k+1}-1}{(s^{k+1}-1)^{\log_s p}},
\]
and $b_{s^{k+1}-1}=b_{[(s-1)^{k+1}]_s}$ is decreasing with respect to $k$, which follows from Proposition \ref{pro mon}. Then
\[
b_{s^{k+1}-1}\geq \lim _{k\to\infty}b_{s^{k+1}-1}=\frac{q(s-1)+r}{p-1}.
\]

\emph{Acknowledgments.} The work is supported by the Fundamental Research Funds for the Central University (Grant Nos. 2662020LXPY010).

\end{document}